\documentclass{article}
\usepackage[utf8]{inputenc}
\usepackage[T1]{fontenc}
\usepackage[a4paper, total={6in, 10in}]{geometry}

\usepackage{amsmath}
\usepackage{amssymb}
\usepackage{amsfonts}
\usepackage{amsthm}

\usepackage{graphicx}
\usepackage{float}
\usepackage{cite}
\usepackage{url}
\usepackage{hyperref}

\usepackage{algorithm}
\usepackage{algorithmic}

\theoremstyle{plain}
\newtheorem{theorem}{Theorem}[section]         
\newtheorem{corollary}{Corollary}[section]     

\theoremstyle{definition}

\theoremstyle{remark}

\begin{document}

\title{Abstract}
\author{Sneha R}
\date{May 2026}

\title{Abstract}
\author{Sneha R}
\date{May 2026}

\title{ Dominated Coloring of Some Classes of Perfect and Unicyclic Graphs}

 \author{Sneha R, J. Geetha \  and K. Somasundaram}
\date{\small{Department of Mathematics, Amrita School of Physical Sciences \\ Amrita Vishwa Vidyapeetham,
  Coimbatore\\ Tamil Nadu, India.\\
 r\_sneha@cb.students.amrita.edu, \{j\_geetha, s\_sundaram\}@cb.amrita.edu}}
\maketitle

\begin{abstract}
 The dominated coloring of a graph $G$ is a proper vertex coloring such that each color class is dominated by a vertex. The minimum number of colors required for a dominated coloring of a graph $G$ is called the dominated chromatic number of the graph $G$ and is denoted by $\chi_{dom}(G)$. A graph $G$ is said to be chromatic dominated if $\chi(G)=\chi_{dom}(G)$. In this paper, we characterized the perfect graphs, which are chromatic dominated. Also, we studied the dominated chromatic number of some classes of unicyclic graphs.
\end{abstract}


\section{Introduction}
Let $G$ be a simple, connected, undirected graph with vertex set $V(G)$ and edge set $E(G)$. A proper vertex coloring of $G$ is an assignment of  colors to vertices of $G$ such that adjacent vertices are assigned different colors. A set of vertices assigned the same color is called a color class. The minimum number of colors among all proper vertex colorings of $G$ is called the chromatic number of the graph $G$ and is denoted by $\chi(G)$. 

A subset $S$ of $V(G)$ is said to be a dominating set of $G$ if all vertices that are not in $S$  have an edge with a vertex in $S$. If every vertex of $G$ has an edge with a vertex in $S$, then $S$ is said to be a total dominating set. The minimum cardinality among all dominating sets (total dominating sets) is known as the domination number (total domination number) of $G$ and is denoted by $\gamma(G)$ ($\gamma_{t}(G)$). A subset $S$ of $V(G)$ is set to be dominated by a vertex $v$, if all vertices of $S$ have an edge with $v$.

The proper vertex coloring of $G$ in which each color class is dominated by a vertex is said to be a dominated coloring of $G$. The minimum number of colors required for a dominated coloring of a graph $G$ is said to be the dominated chromatic number of $G$ and is denoted by $\chi_{\text{dom}}(G)$. The concept of dominated coloring was introduced by Merouane et al. \cite{merouane2015dominated}. They established lower and upper bounds for the dominated chromatic number and characterized the case of equality of the lower bound. They proved that the dominated coloring problem is NP-complete for arbitrary graphs with $\chi_{dom}\geq 4$ and also gave a polynomial time algorithm for recognizing graphs with $\chi_{dom}\leq 3$. A possible application of the dominated coloring problem is the development of interpersonal relationships in the social network, and it is discussed in \cite{chen2014dominated}. Fatemeh Choopani et al. \cite{choopani2018dominated} investigated the corona of two graphs and proved that Vizing-type conjecture holds for dominated colorings of the direct product of two graphs. They also obtained the dominated chromatic number of Mycielsky graph of $G$ in terms of $\chi_{dom}(G)$ and some Nordhaus-Gaddam-type results for the dominated chromatic number. Alikhani and Piri \cite{alikhani2019dominated} examined the effects on the dominated chromatic number of a graph when the graph is modified by operations on vertices and edges. They \cite{alikhani2019_1dominated} also obtained the dominated chromatic number of certain graphs. Fatemeh Choopani et al. \cite{choopani2022dominated} obtained the dominated chromatic number of total and middle graphs of paths and cycles. The dominated coloring of Cartesian product, direct product, lexicographic product, and strong product of some graphs are discussed in \cite{li2023dominated}. The dominated chromatic number of some networks, including circulant, Sierpinski, chain silicate, and cyclic silicate, is studied in \cite{poonkuzhali2024dominated} and that of several tree structures is studied in \cite{poonkuzhali2025dominated}. Azeddine et al. \cite{benkaci2025dominated} obtained a new bound for the dominated chromatic number and determined graphs attaining some of those bounds. They also characterized the cubic graphs whose dominated chromatic number and chromatic number are the same. 

\section{Chromatic Dominated Graphs}
A graph $G$ is said to be chromatic dominated if $\chi(G)=\chi_{dom}(G)$. It is easy to see that $K_n$ is chromatic dominated where as $P_n$ and $C_n$ are not chromatic dominated. The characterization of graphs into chromatic dominated is an interesting open problem. Azeddine et al. \cite{benkaci2025dominated} proposed the following open problem.\\

\noindent \textbf{Open Problem:}\cite{benkaci2025dominated} Characterization of graphs with $\chi(G)=\chi_{dom}(G)$.\\

In this section, we address the open problem for some classes of perfect graphs. A graph $G$ is said to be perfect if $\chi(G)=\omega(G)$, where $\omega(G)$ is the clique number of $G$. The important subclasses of perfect graphs include block graphs, split graphs, co-bipartite graphs, and $k$-trees. A split graph is a graph whose vertex set can be partitioned into a clique and an independent set. Merouane et al. \cite{merouane2015dominated}  showed that, for any split  graph $G$, $\chi_{dom}(G)=\omega(G)$.Therefore any split graph is chromatic dominated.
\subsection{Co-bipartite Graph}
   A co-bipartite graph is a graph $G$ whose complement is bipartite. Equivalently, $G$ is co-bipartite if its vertex set can be partitioned into two cliques. In other words, there exist disjoint sets $A$ and $B$ such that $A\cup B=V(G)$, $A$ and $B$ form two disjoint cliques. Let $|A|=m$ and $|B|=n$, $m \geq n$. Let $M$ be a maximum matching between $A$ and $B$ and $|M|=k$. Let $S$ be the set of all edges between $A$ and $B$, that are not in $M$. Let $N$ be a subset of $S$ such that each edge in $N$  saturates only one vertex in $M$ and $r=|N|$.\\

In the next theorem, we characterized the co-bipartite graph.
\begin{theorem}
 Let $G$ be a co-bipartite graph. Then   \begin{center}
        $\chi_{dom}(G)=
        \begin{cases}
           m+n-(2k+r), & if \ 2k+r\leq n  \\
           m, & otherwise.
        \end{cases}$
    \end{center}
\end{theorem}    
\begin{proof}
Let $G$ be a co-bipartite graph with $V(G)=A\cup B$.  Let $M$ be a maximum matching in $G$ with $|M|=k$  between $A$ and $B$. 

Since $A$ and $B$ are cliques, we assign $m$ and $n$ colors to $A$ and $B$ respectively. Now consider an edge $e=(uv)\in M$. The edge $e$ allows the reduction of two colors in the dominated coloring of $G$ in such a way that any one of the neighbors of $u$ can be colored with the color of $v$ and any one of the neighbors of $v$ can be colored with the color of $u$. Thus, the $k$ edges together reduce the number of colors required for a dominated coloring of $G$ by $2k$.

Let $N$ be the set of edges between $A$ and $B$ such that each edge in $N$  saturates only one vertex in $M$ and $r=|N|$. Now, consider an edge $e'=(xy)\in N$ saturates $y$, $y\in V(M)$ (in particular, $x\in A$ and $y\in B$. 

When $n-2k\geq r$, we recolor a neighbor vertex of $y$ in $B$ with the color of $x$. Hence, the edge $e'$ allows us to reduce one color from the clique $B$. Similarly, we recolor the $r$ vertices in the clique $B$ since $|N|=r$. Therefore the $k+r$ edges together reduce $2k+r$ colors in $G$. Fig. \ref{cobipartite1} shows a graph with $n-2k\geq r$. Hence,    $\chi_{dom}(G)=m+n-(2k+r)$.

If $n-2k<r$, every vertex in $B$ is already recolored with colors used in $A$, since $2k+r<n$. In this case the  dominated coloring of $G$ uses exactly $m$ colors. Fig. \ref{cobipartite2} shows a graph with $n-2k\geq r$. Therefore, we have $\chi_{dom}(G)=m$.

\begin{figure}[ht]
    \centering
    \begin{minipage}{0.45\textwidth}
        \centering
        \includegraphics[width=\textwidth]{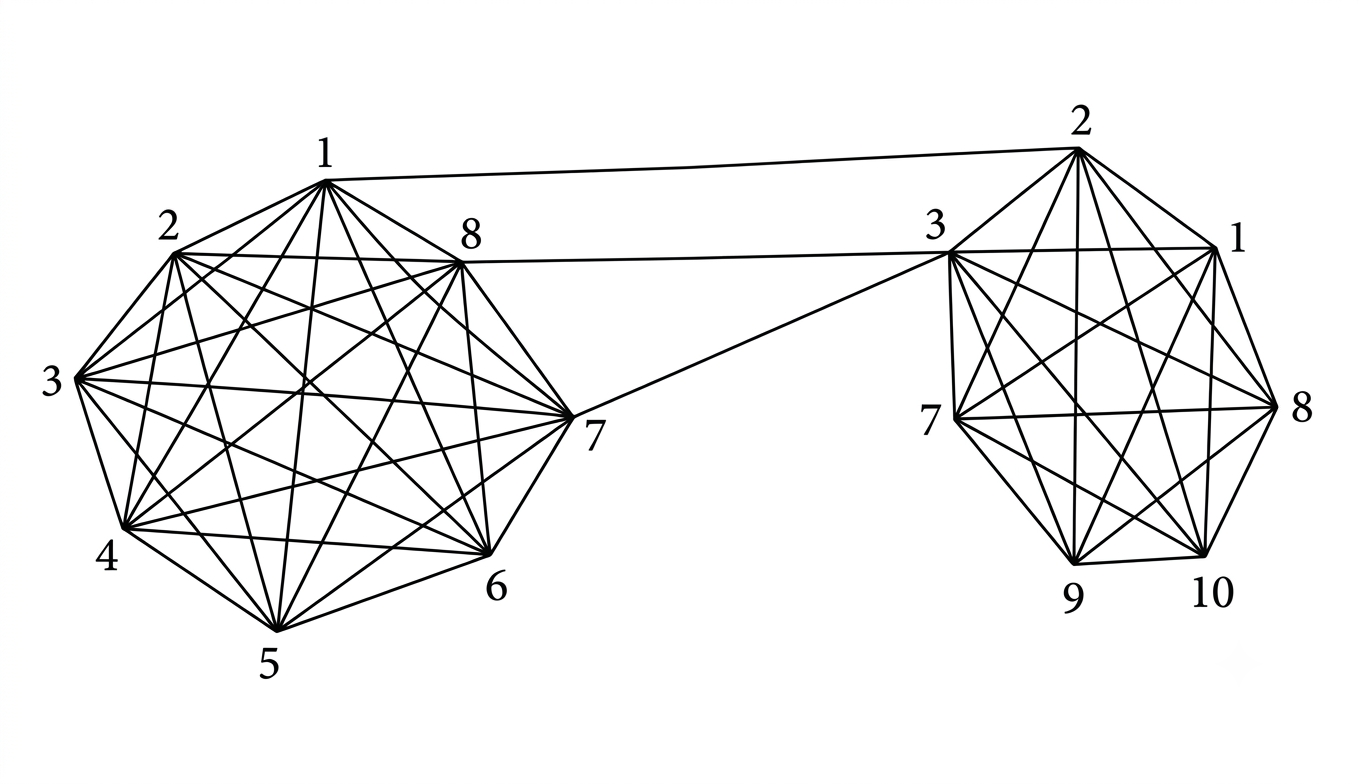}
        \caption{Co-bipartite graph with $n-2k\geq r$.}
        \label{cobipartite1}
    \end{minipage}
    \hfill
    \begin{minipage}{0.51\textwidth}
        \centering
        \includegraphics[width=\textwidth]{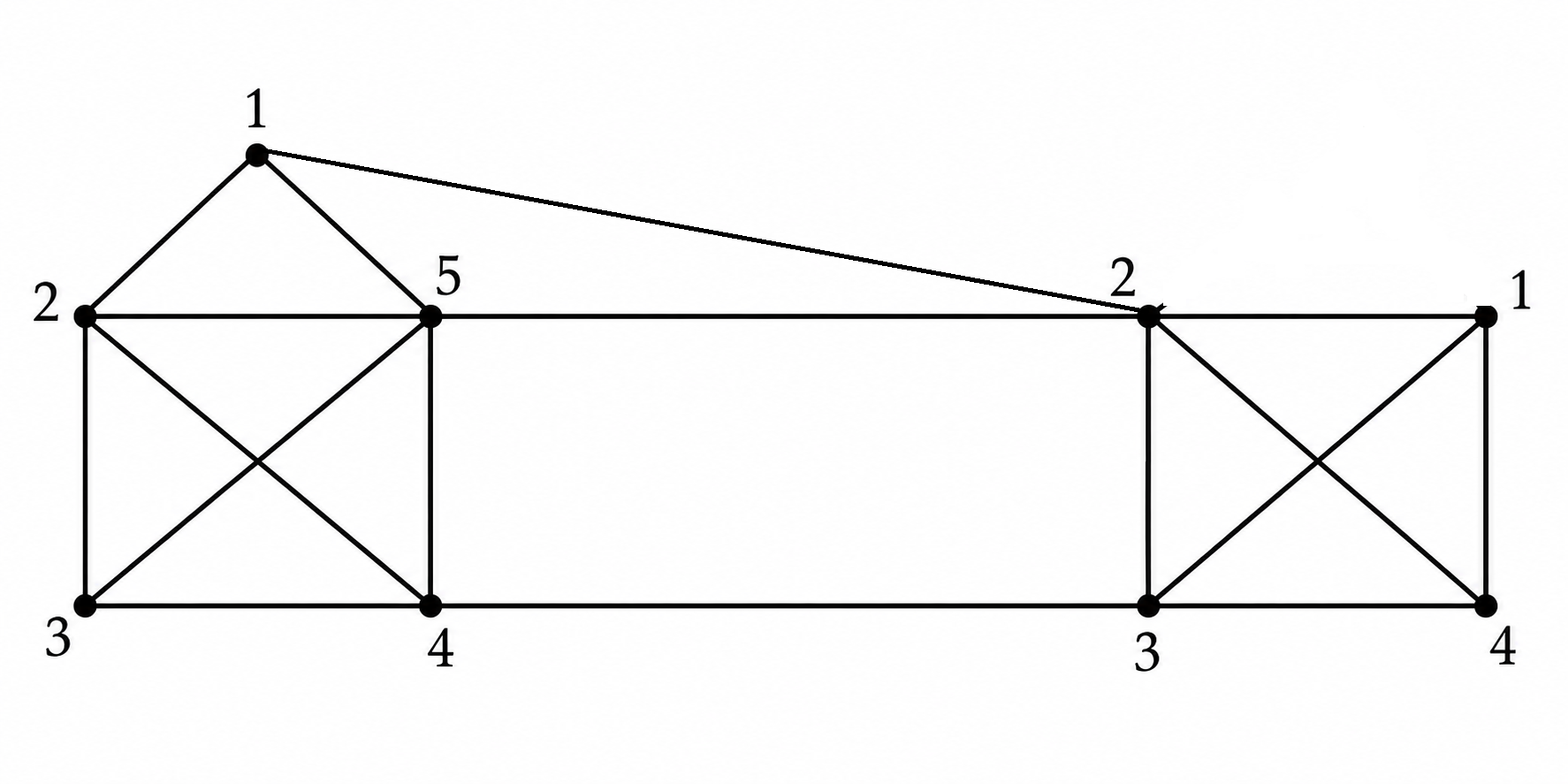}
        \caption{Co-bipartite graph with $n-2k<r$.}
        \label{cobipartite2}
    \end{minipage}
\end{figure}
    
\end{proof}   
\subsection{Dominated Coloring of Block Graphs}
A block is a maximal bi-connected component of a graph. A block graph is a graph in which every block is a clique \cite{harary1963characterization}. Block graphs are another important class of perfect graphs. In the next theorem, we resolve the open problem for the block graphs. 

We construct a  family $\mathcal{F}$ of block graphs  as follows: Consider a clique $C$ with $n$ vertices and attach cliques $C_i^j, \ j=0,1,\dots,m$ with the vertex $i, i=1,\dots,n$, such that $\sum\limits_i^n (\omega_i-1)\leq n$, where $\omega_i=\omega(C_i^j)$. 
\begin{theorem}
    A block graph $G$ is chromatic dominated if and only if $G\in\mathcal{F}$.
\end{theorem}
\begin{proof}
Let $V(C)=\{1,\dots,n\}$, we color the clique $C$ with $n$ colors. We take $\omega_i-1$ colors from the $n$ colors of $C$ and color the satellite cliques $C_i^j$ attached in the vertex $i$ (except the vertex $i$). The maximum number of colors taken from $V(C)$ to color the satellite cliques is $\sum\limits_i^n (\omega_i-1)$. As $\sum\limits_i^n (w_i-1)\leq n$, we required $n$ colors for the graph $\mathcal{F}$.\\

\noindent Conversely, suppose $G$ is chromatic dominated and $G\notin\mathcal{F}$.\\

\noindent Case (1): If the cliques are not attached as per the construction of $\mathcal{F}$.\\

In this case, at least one clique, say $Q$, is not attached with $C$. We need more colors to color the clique $Q$ for the domination. Hence $\chi_{dom}(G)>\omega(G)$, which is a contradiction.\\

\noindent Case (2): $\sum\limits_i^n (w_i-1) > n$.\\
In this case, we need more than $n$ colors to color the satellite cliques $C_i^j$ while satisfying the domination condition. Hence $\chi_{dom}(G)>n$, which is a contradiction.

\end{proof} 
We prove that if all the cliques share a vertex in a block graph then the graph is chromatic dominated. 
\begin{corollary}\label{coro}
 Let $G$ be a block graph such that all the blocks share a common vertex. Then         $\chi_{dom}(G)=\omega(G)$.
\end{corollary}
\begin{proof}
Let $\omega(G)=m$ and $u$ be the common vertex of all the $j$ blocks. Color all the blocks with $m$ colors. The vertex $u$ will dominate all color classes in $C_u^j$. Hence $\chi_{dom}(G)=\omega(G)$.
\end{proof}

Chain block graphs are connected block graphs with their blocks arranged in a linear sequence. If all the blocks are of the same size, the graph is called a uniform chain block graph. Let $K_m^n$ be a uniform chain block graph with $n$ blocks, $n\geq2$, each of size $m$. 
\begin{theorem}
\[\chi_{dom}(K_m^n)=\begin{cases}
            \frac{n(m-1)}{2}+1, & if\ n\ is\ even \\
            \frac{(n+1)(m-1)}{2}, & if \ n\ is\ odd.
        \end{cases} \]
\end{theorem}
\begin{proof}
Note that the graph $K_m^n$ has $m-1$ more vertices than $K_m^{n-1}$. It is easy to see that, due to the domination constraint, any dominated coloring of $K_m^n$ allows the same color to appear in at most two blocks, and these blocks must be consecutive.
For $n=1$, $K_m^n$ is a complete graph, and for $n=2$, it is easy to see from the theorem 2.2 that $K_m^2$ is chromatic dominated graph. Therefore, $\chi_{dom}(K_m^1)=\chi_{dom}(K_m^2)=m$.

\noindent For $n>2$, we prove by using mathematical induction on the number of blocks $n$.\\

\noindent \textbf{Claim:}\[   \chi_{dom}(K_m^n)=\begin{cases}
            \chi_{dom}(K_m^{n-1})+m-2, & if\ n\ is\ odd \\
            \chi_{dom}(K_m^{n-1})+1, & if \ n\ is\ even.
        \end{cases} \]
        
For $n=3$, consider a dominated coloring of $K_m^2$ using $m$ colors. In this coloring, $m-1$ vertices of the second block use colors from the first block. The color of the common vertex between the first and second blocks can also be used for one of the $m-1$ vertices in the third block, and the remaining $m-2$ vertices in the third block must  receive $m-2$ new colors. Therefore, an optimum dominated coloring of $K_m^3$ requires exactly $\chi_{\text{dom}}(K_m^2) + 1$ colors. \\

For $n=4$, the $m-2$ new colors introduced in the third block during the dominated coloring of $K_m^3$ can also be used in the fourth block for the dominated coloring of $K_m^4$. This is possible as long as the shared vertex between the third and fourth blocks should be given the same color as the common vertex of the first and second blocks. This coloring leaves exactly one vertex uncolored, which must be assigned a new color. Therefore,
$\chi_{dom}(K_m^4) = \chi_{dom}(K_m^3) + 1$.
It follows that the hypothesis is valid for the base cases.

Assume that the induction hypothesis holds for all integers up to $n=k$.\\

Consider a dominated coloring of $K_m^k$.
If $k+1$ is even, then $k$ is odd. By the induction hypothesis, the $k^{\text{th}}$ block of $K_m^k$ contains $m - 2$ colors that are not shared with any of the preceding blocks. In the dominated coloring of $K_m^{k+1}$, $m - 2$ of the $m - 1$ additional vertices can be assigned the same $m - 2$ colors used in the $K_m^k$ block, while the remaining vertex must receive a new color. Therefore, $\chi_{dom}(K_m^{n+1})=\chi_{dom}(K_m^n)+1$.\\

When $k+1$ is odd, $k$ is even. By the induction hypothesis, the $n^{\text{th}}$ block contains only one color not shared with any preceding block. This color can be assigned to one vertex in the $(k+1)^{\text{th}}$ block of $K_m^{k+1}$ in its dominated coloring, while the remaining $m - 2$ vertices require new colors. Therefore,
     $\chi_{dom}(K_m^{k+1})=\chi_{dom}(K_m^k)+m-2$. Hence the claim.\\
     
From the claim, it is easy to see that \begin{center}
        $\chi_{dom}(K_m^n)=\begin{cases}
            \frac{n(m-1)}{2}+1, & if\ n\ is\ even \\
            \frac{(n+1)(m-1)}{2}, & if \ n\ is\ odd.
        \end{cases}$
    \end{center}
\end{proof}  

Let $\Gamma$ be a block graph decomposed into subgraphs $S_i$, $D_j$ and $L_k$ as follows: Each $S_i$ is a maximal subgraph with structure similar to the graph in corollary\ref{coro} . The $D_j$ are single blocks located between two $S_i$'s, while $L_k$'s are end blocks-blocks attached to only one $S_i$.
\begin{theorem}
 \[\chi_{dom}(\Gamma)\leq min\{\sum_i\omega(S_i)+\sum_j(\omega(D_j)-2)+\sum_k(\omega(L_k)-1)\}, \]
    where the minimum is taken over all possible decompositions of $\Gamma$.
\end{theorem}

\subsection{Dominated Coloring of \texorpdfstring{$k$}{k}-trees}

A $k$-tree is a graph that can be reduced to the $k$-complete graph by a sequence of removals of a degree $k$ vertex with completely connected neighbors. A $k$-tree with $k+1$ vertices is called
base clique. 

Merouan et al. \cite{merouane2015dominated} proved that $
\chi_{dom}(P_n) = 
\begin{cases} 
\left\lfloor\frac{n}{2}\right\rfloor & \text{if } n \equiv 0 \pmod 4 \\ 

\left\lfloor\frac{n}{2}\right\rfloor + 1, & \text{otherwise.}
\end{cases}$, and therefore   $\chi_{dom}(P_n)\neq \chi(P_n)$. Hence $P_n$ is not chromatic dominated. We know that any tree is a 1-tree, and hence, in general, any $k$-tree is not chromatic dominated.
\begin{theorem}\cite{merouane2015dominated}\label{k-tree}
    Let $G$ be a graph with maximum degree $\Delta$ and order $n$. Then $\chi_{dom}(G)\geq \frac{n}{\Delta}$.
\end{theorem}
The following theorems give a necessary condition for a $k-$tree to be chromatic dominated.
\begin{theorem} If $G$ is a chromatic dominated $k$-tree with $n$ vertices then $\frac{n}{k+1}\leq \Delta(G)$, where $\Delta(G)$ is the maximum degree of $G$.
\end{theorem}
\begin{proof}
    If $G$ is chromatic dominated, $\chi_{dom}(G)=k+1$. It follows from theorem \ref{k-tree} that $\frac{n}{k+1}\leq \Delta(G)$.
\end{proof}

The converse of the above theorem is not true. For example, consider a graph $G$ shown in Fig. \ref{k-tree_1}. The graph $G$ is not chromatic dominated since $\chi(G)=2\neq \chi_{dom}(G)=3$, but the condition is satisfied.
\begin{figure}[H]
    \centering
    \includegraphics[width=0.3\linewidth]{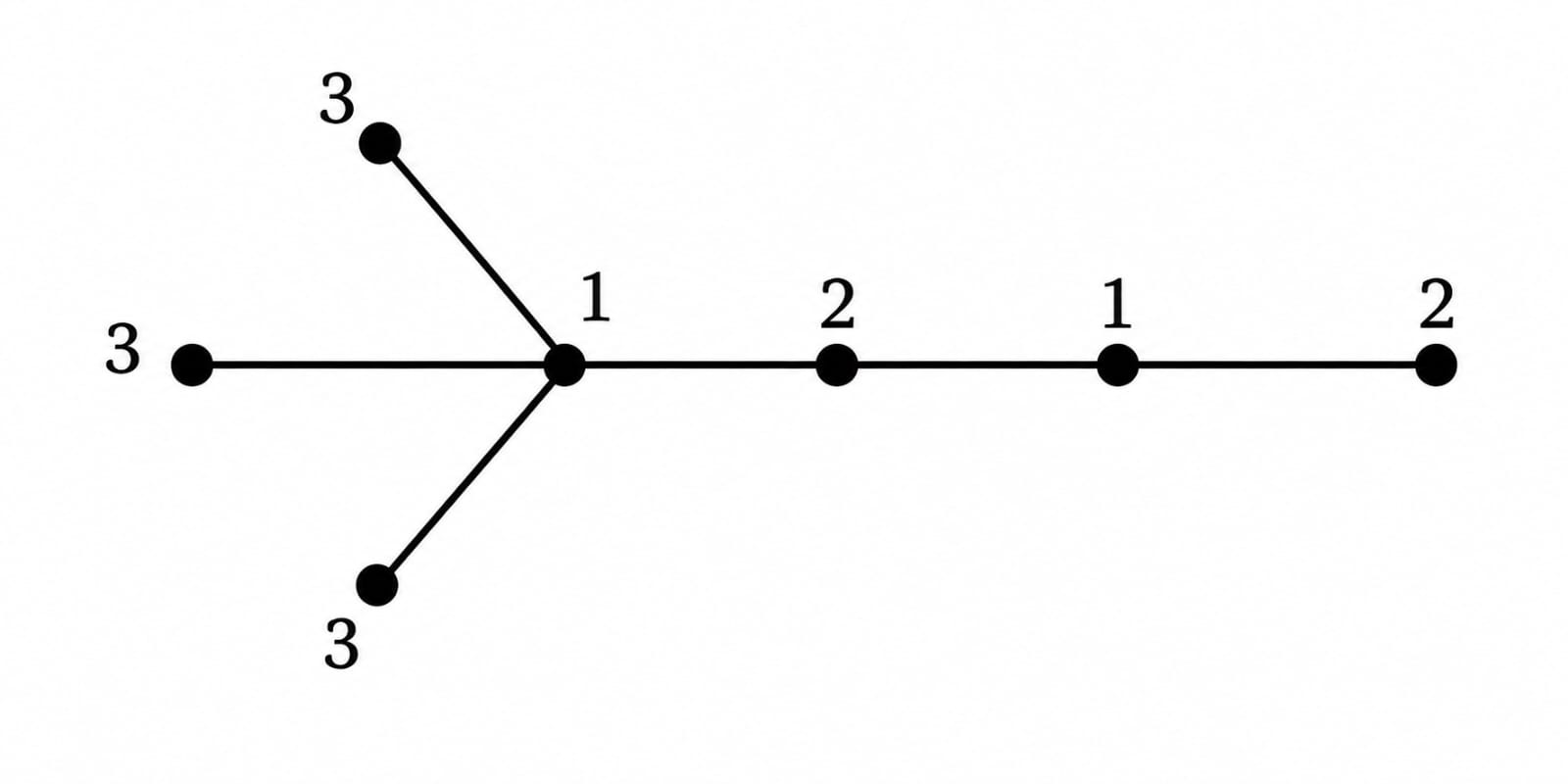}
    \caption{Example for the converse of the Theorem 2.6.}
    \label{k-tree_1}
\end{figure}
\begin{theorem}\cite{benkaci2025dominated}
    Let $G$ be a connected graph of order $n$ and with diameter $diam(G)$. Then $\frac{1}{2}[diam(G)+1]\leq\chi_{dom}(G)\leq n-\frac{1}{2}[diam(G)-1]$.
\end{theorem}
\begin{theorem}
    If $G$ is a chromatic dominated $k$-tree with order $n$, then the diameter of $G$  $diam(G)\leq min\{2k+1,2(n-k)+1\}$.
\end{theorem}
\begin{proof}
    If $G$ is chromatic dominated, $\chi_{dom}(G)=k+1$. It follows from the above theorem that $diam(G)\leq 2k+1$ and $diam(G)\leq 2(n-k)+1$. Hence $diam(G)\leq min\{2k+1,2(n-k)+1\}$.
\end{proof}
The condition given in the above theorem is not sufficient. The 2-tree $G$ shown in Fig.\ref{k-tree-3} is not chromatic dominated, but  $diam(G)\leq min\{2k+1,2(n-k)+1\}$.

\begin{figure}[H]
    \centering
    \includegraphics[width=0.30\linewidth]{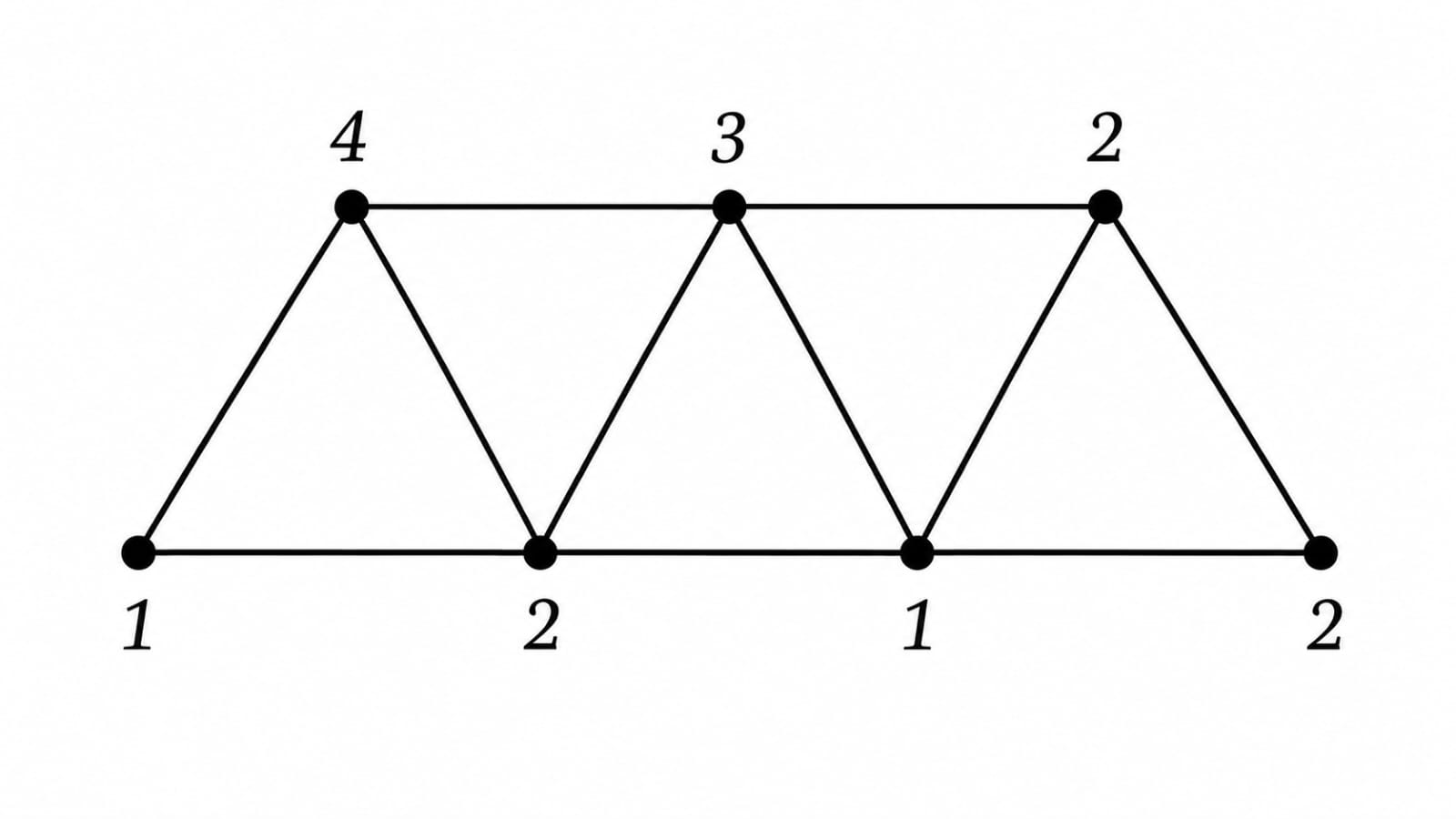}
    \caption{Example for the converse of Theorem 2.8.}
    \label{k-tree-3}
\end{figure}

The following theorem gives a sufficient condition for a $k$-tree to be chromatic dominated.
\begin{theorem} Let G be a k-tree with n vertices such that all k-cliques have a common vertex. Then, $G$ is chromatic dominated.
\end{theorem}
\begin{proof}
Let $u$ be the common vertex to all cliques $k$ in $G$. We use mathematical induction on the number of vertices $n$ to prove the theorem. \\

Suppose the graph $G$ is a clique (base clique) of size $n=k+1$ then we know that $\chi_{dom}(G)=\omega(G)=\chi(G)$. \\

Let us assume that the hypothesis is true for a $k$-tree $G'$ with $n=k+l$ vertices.\\

Now, we construct a $k$-tree $G''$ with $n=k+l+1$ vertices by adding a new vertex $v$ with $G'$, and all the $k$ cliques in $G''$ having a common vertex $u$. 
Consider an optimum dominated coloring of $G'$,  and we extend it to $G''$ by retaining the same colors for all vertices of $G'$. One color is missing at $v$, say $c$, and we use $c$ to $v$. This coloring is proper, and every color class of $G''$ remains dominated by the same vertex as in $G'$, while the color class containing $v$ is dominated by $u$. Hence, this coloring is a dominated coloring. Therefore $\chi(G'')=\omega(G'')=\chi_{dom}(G'')$. 
\end{proof}
The converse of the above theorem is also not true. The graph given in Fig. \ref{k-tree_2} is a chromatic dominated 2-tree, but there is no common  vertex to all 2-cliques.
\begin{figure}[H]
    \centering
    \includegraphics[width=0.25\linewidth]{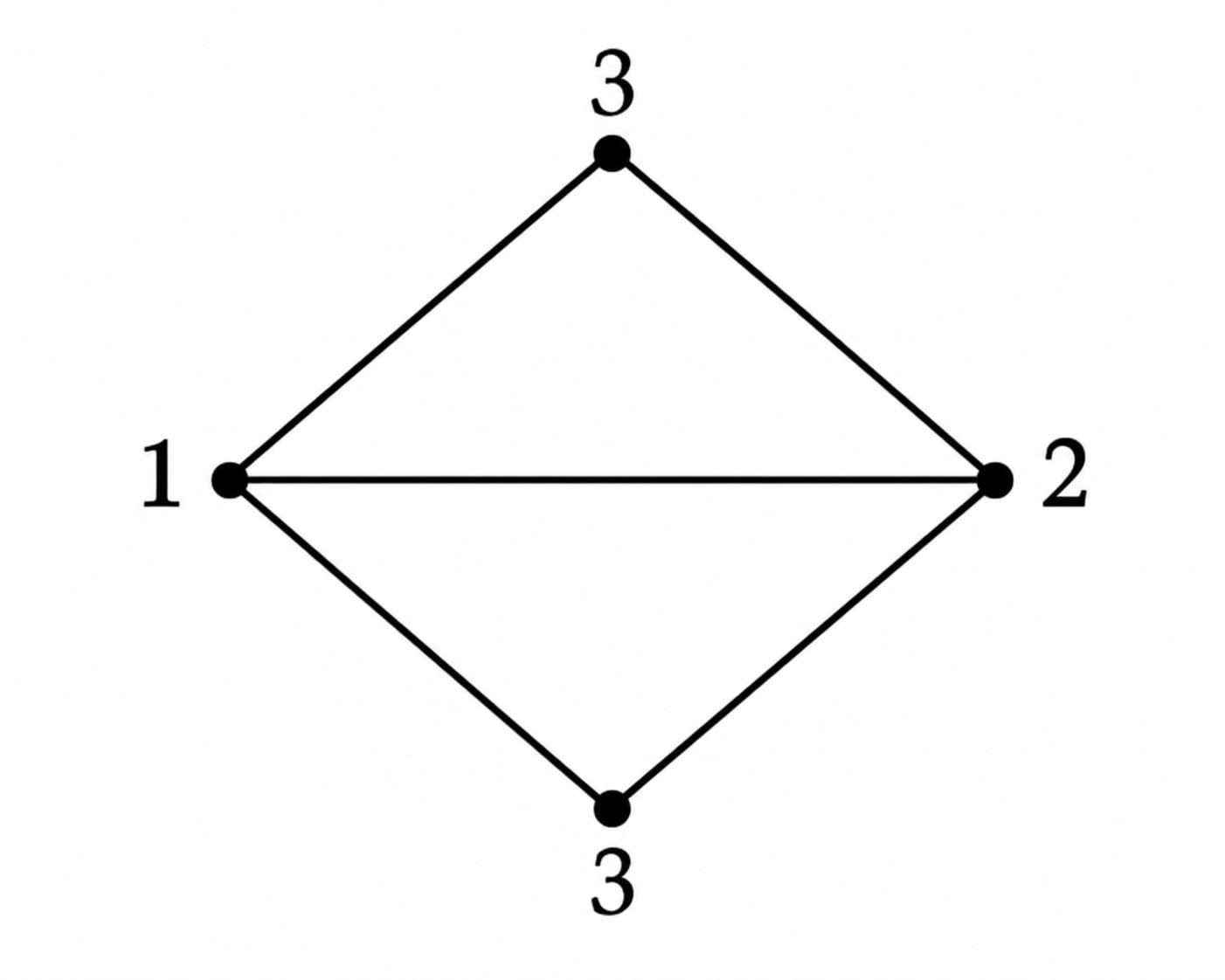}
    \caption{Example for the converse of Theorem 2.9.}
    \label{k-tree_2}
\end{figure}

From the above theorem, we observed that for any positive integer $t$, there exists a $k$-tree $G$ with $\chi_{dom}(G)=k+t.$ 

\textbf{Open Problems:}
\begin{enumerate}
    \item For any positive integer $t$, there exists a $k$-tree $G$ with $\chi_{dom}(G)=k+t.$
    \item For every $k$-tree, there exists a positive integer $l$ such that $\chi_{dom}(G)=k+l.$ \\
\end{enumerate}

These two open problems are true for 1-tree. For example, for any positive integer, we can find a $P_n$ (1-tree) such that $\chi_{dom}(P_n)=t+1.$ Similarly, for any $n$, we can find an integer $l$ such that $\chi_{dom}(P_n)=l+1.$ The two problems remain open for $k\geq 2.$

\section{Unicyclic Graphs}
Unicyclic graphs are connected graphs containing exactly one cycle. A connected graph $G$ is unicyclic if and only if the number of vertices and edges are the same. Unicyclic graphs have a number of subclasses, including sun graphs, generalized sun graphs, and $m,n$-Tadpole graphs.\\

The $(m,n)$-Tadpole graph $T_{m,n}$ is a unicyclic graph that is formed by joining an arbitrary vertex of the cycle graph $C_m$ and an end vertex of the path graph $P_n$ by an edge. Such graphs have $m+n$ vertices and $m+n$ edges. 

\begin{theorem}\cite{merouane2015dominated}\label{uni2}
    \[
\chi_{dom}(C_n) = \chi_{dom}(P_n)=
\begin{cases} 
\left\lfloor\frac{n}{2}\right\rfloor & \text{if } n \equiv 0 \pmod 4 \\ 

\left\lfloor\frac{n}{2}\right\rfloor + 1, & \text{otherwise.}
\end{cases}
\]
\end{theorem}

The total domination number of $(m,n)$-tadpole graphs was determined by Michael A. Henning\cite{henning2000graphs}. Merouane et.al\cite{merouane2015dominated} showed that $\chi_{dom}(G)=\gamma_t(G)$ for any triangle-free graph $G$. Since $(m,n)$-tadpole graphs are triangle-free for $m>3$, their dominated chromatic number can be determined directly from the above result. We provide an alternative proof of this result in the following theorem. 
\begin{theorem} 
\[
\chi_{dom}(T_{m,n}) = 
\begin{cases} 
\left\lfloor\frac{n}{2}\right\rfloor + \left\lfloor\frac{m}{2}\right\rfloor, & \text{if } n \equiv 0,1 \pmod 4 \text{ and } m \equiv 0 \pmod 4 \\ 
& \text{or if } n \equiv 2 \pmod 4 \text{ and } m \not\equiv 1 \pmod 4 \\[1.5ex]
\left\lfloor\frac{n}{2}\right\rfloor + \left\lfloor\frac{m}{2}\right\rfloor + 1, & \text{otherwise.}
\end{cases}
\]
\end{theorem}
\begin{proof}
    Let $T_{m,n}$ be the tadpole graph obtained by joining the cycle $C_m$ and the path $P_n$ by a bridge. Let $V(C_m)=\{x_1,x_2,.\ .\ .\ , x_m\}$ and $V(P_n)=\{y_1,y_2\ .\ .\ .\ ,y_n\}$. The vertices of $T_{m,n}$, $V(T_{m,n})=V(C_m)\cup V(P_n)$ and   $E(G)=E(C_m)\cup E(P_n)\cup \{x_ky_n\}$. Fig. \ref{tadpole} shows an example of $T_{m,n}$.
\begin{figure}[H]
    \centering
    \includegraphics[width=0.5\linewidth]{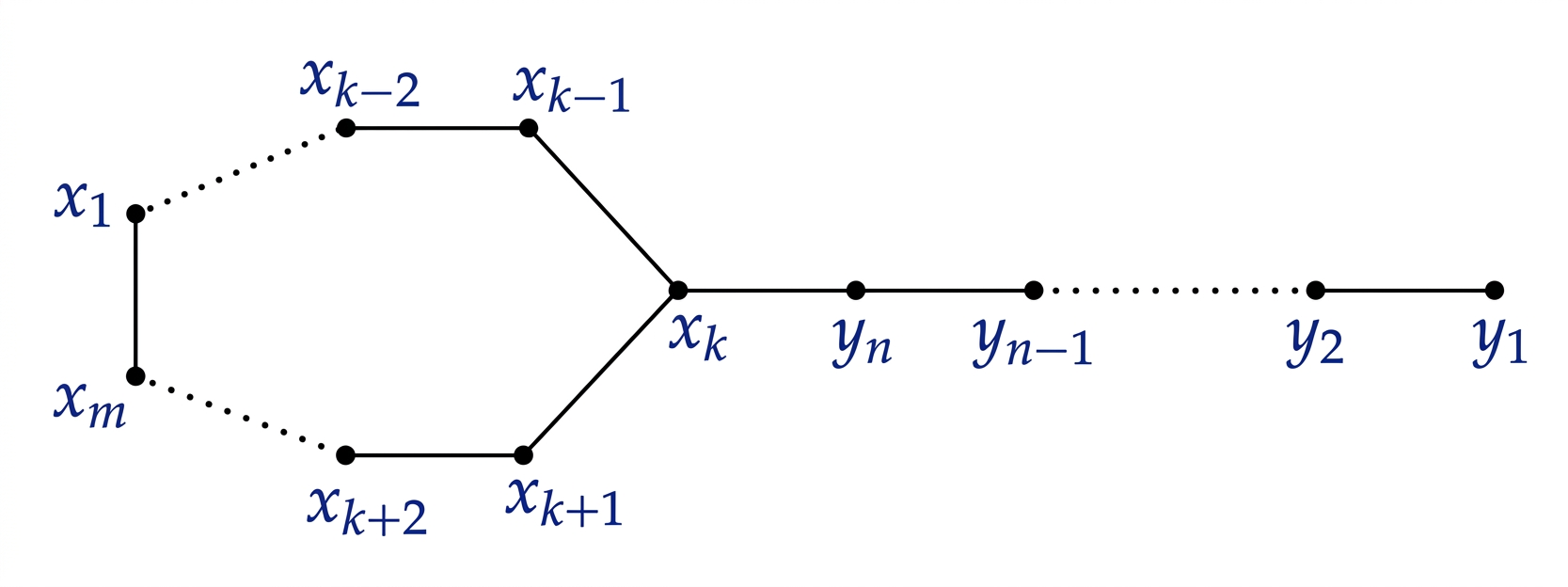}
    \caption{$T_{m,n}.$}
    \label{tadpole}
\end{figure}

Let $f_c:V(C_m)\to \{a_1,a_2,a_3,.\ . \ .,a_r\} $ and $f_p:V(P_n)\to\{b_1,b_2,. \ . \ . \ ,b_s\} $ be a dominated coloring of $C_m$ and $P_m$ respectively, where $r=\chi_{dom}(C_m)$ 
  and $s=\chi_{dom}(P_n)$. Define $f:V(T_{m,n})\to \{a_1,\dots, a_r,b_1,\dots,b_s\}$ such that 
 \begin{center}
     $f(v)=\begin{cases}
         f_c(v), &\text{if}\ v\in V(C_m)\\
         f_p(v), & \text{if} \ v\in V(P_n).
     \end{cases}$
 \end{center}
 It is easy to see that $f$ is a dominated coloring of $T_{m,n}$ with at most $r+s$ colors.
 
 That is, $\chi_{dom}(T_{m,n})\leq \chi_{dom}(C_m)+\chi_{dom}(P_n).$

The four vertices $x_{k-1},x_{k+1},x_k$, and $y_n$ form a $K_{1,3}$ and the vertices $x_k,y_n$ and $y_{n-1}$ form a $K_{1,2}$. We recolor the vertices $x_{k-1},x_{k+1}$ with  the color of $y_n$ and similarly we recolor the vertex $x_k$ with the color of $y_{n-1}$.  In this recoloring, we reduce the coloring of $T_{m,n}$. We consider different cases of $m$ and $n$. 

 \noindent For $n\equiv 0\pmod{4}$ and any $m$,  $n\equiv3\pmod{4},m\equiv0\pmod{4}$ we use the same coloring as in $f$. 
 For $n\equiv 1\pmod{4}$ and $m$, and $ n\equiv 2\pmod{4},m\equiv0,1\pmod{4}$, we reduce one color from the coloring $f$.
 For the remaining case, we reduce two colors from the coloring $f$.

\noindent Therefore the dominated chromatic number of $T_{m,n}$ is given by
\begin{center}
    $\chi_{dom}(G)=\begin{cases}
        \chi_{dom}(C_m)+\chi_{dom}(P_n), & \text{if}\ n\equiv 0\pmod{4}, \text{any} \ m\ \text{and } \\ &n\equiv3\pmod{4},m\equiv0\pmod{4}\\
        \chi_{dom}(C_m)+\chi_{dom}(P_n)-1, & \text{if}\ n\equiv 1\pmod{4},\text{ any}\ m\ \text{and}\\ 
        &n\equiv 2\pmod{4},m\equiv0,1\pmod{4}\\
        \chi_{dom}(C_m)+\chi_{dom}(P_n)-2, & \text{if}\ n\equiv2\pmod{4},m\equiv2,3\pmod{4}.
    \end{cases}$
\end{center}
Theorem \ref{uni2}, will give the desired results.
\end{proof}

A $m$-Pan graph is $(m,n)$-Tadpole graph when $n=1$.

\begin{corollary}
    The dominated chromatic number of $m$-Pan graph $G$ is
    \begin{center}
        $\chi_{dom}(G)=\begin{cases}
            \frac{m}{2}, & \text{if} \ m\equiv0\pmod{4}\\
            \lfloor\frac{m}{2}\rfloor+1, & \text{otherwise}.
        \end{cases}$
    \end{center}
\end{corollary}

A Sun graph is an unicyclic graph, obtained by attaching a pendant vertex to each vertex of a cycle graph $C_m$ of order $m$ and it is denoted by $S_m$. More generally, a path graph $P_n$ of order $n$ is attached to every vertex of $C_m$ by a bridge; the resulting graph is called a generalized sun graph, denoted by $S_m^n$. Various classes of generalized sun graphs can be constructed by attaching different types of trees randomly to the vertices of the underlying cycle. We are considering two such classes.

\begin{theorem}
    Let $G$ be a unicyclic graph formed by attaching $k$ paths $P_n$ at every vertex of the cycle $C_m$ by $k$ edges. Then the dominated chromatic number of $G$ is given by
    \begin{center}
        $\chi_{dom}(G)=\begin{cases}
            \frac{m}{2}(kn+1), & \text{if}\ m,n\equiv 0\pmod{4}\\
            \frac{kmn}{2}+\lfloor\frac{m}{2}\rfloor+1, & \text{if}\ n\equiv0\pmod{4}\ \text{and}\ m\not\equiv0\pmod{4}\\
            m[\frac{k(n-1)}{2}+1], & \text{if}\ n\equiv 1\pmod{4}\ \text{and any}\ m\\
            m[\frac{kn}{2}+1], & \text{if}\ n\equiv 2\pmod{4}\ \text{and any}\ m\\
            \frac{km(n+1)}{2}, &  \text{if}\ n\equiv 3\pmod{4}\ \text{and any}\ m.
        \end{cases}$ 
    \end{center}
\end{theorem}
\begin{proof}
    Let $\{x_1,x_2,.\ .\ .\ ,x_m\}$ be the vertex set of the cycle $C_m$ in $G$. For $1\leq i\leq m$ and $1\leq j\leq k$, let $P_{ij}$ denote the $j^{th}$ path attached to $x_i$. Let the vertex set of $P_{ij}$ be $\{y_{ij}^1,y_{ij}^2,.\ .\ .\ y_{ij}^n\}$, where $x_iy_{ij}^1\in E(G)$. Let $f_c$ be an optimum dominated coloring of $C_m$ and for each $i$ and $j$, $f_{ij}$ be an optimum dominated coloring of the $j^{th}$ path attached at $x_i$. Then $f$ defined by
    \begin{center}
        $f(v)=\begin{cases}
            f_c(v), & if\ v\in V(C_m)\\
            f_{ij}(v), & if\ v\in V(P_{ij}).
        \end{cases}$
    \end{center}
    is a dominated coloring of $G$ that uses $km\chi_{dom}(P_n)+\chi_{dom}(C_m)$. 
    For each $i$, the vertices \\
    $y_{i1}^1y_{i2}^1,.\ .\ .\ y_{ik}^1,x_{i-1},x_{i+1}$ and $x_i$ form a $K_{1,k+2}$. Therefore, for each $i, j>1$, we can recolor the vertices $y_{ij}^1 x_{i-1}$ and $x_{i+1}$ with the color of $y_{i1}^1$. This recoloring does not reduce the number of colors required for $n\equiv0\pmod{4}$. However for $n\equiv1,2\pmod{4}$, this recoloring makes the coloring optimal and reduces the number of colors by $\chi_{dom}(C_m)+m(k-1)$. When $n\equiv3\pmod{4}$, for each $i$, the color of $x_i$ can be replaced by the color of one of the vertices in $\{y_{i1}^2,y_{i2}^2,. \ .\ .\ ,y_{ik}^2\}$. This reduces the number of colors by $\chi_{dom}(C_m)$. Hence, the dominated chromatic number of $G$ is given by,
    \begin{center}
        $\chi_{dom}(G)=\begin{cases}
            \chi_{dom}(C_m)+km\chi_{dom}(P_n), & \text{if}\ n\equiv0\pmod{4}\\
            km\chi_{dom}(P_n)-m(k-1), &\text{if}\ n\equiv1\ \text{or} \ 2\pmod{4}\\
            km\chi_{dom}(P_n),& \text{if} n\equiv 3\pmod{4}.
        \end{cases}$
    \end{center}
    Theorem \ref{uni2} yield the desired result.
\end{proof}
\begin{corollary}
    The dominated chromatic number of the generalized sun graph $S_m^n$ is
    \begin{center}
        $\chi_{dom}(S_m^n)=
        \begin{cases}
            \frac{m(n+1)}{2} & ,\text{if}\ m,n\equiv0\pmod{4} \\
            \frac{mn}{2}+\lfloor\frac{m}{2}\rfloor+1 &, \text{if}\ m\not\equiv 0\pmod{4}\ \text{and} \ n\equiv0\pmod{4}\\
            m(\lfloor\frac{n}{2}\rfloor+1) & \text{if}\  n\not\equiv0\pmod{4}\ \text{and any} \ m .
        \end{cases}$
    \end{center}
\end{corollary}
\begin{proof}
The graph $S_m^n$ can be obtained as a special case of the graph considered in the previous theorem by taking $k=1$.
\end{proof}
\begin{corollary}
    The dominated chromatic number of the sun graph $S_m$ is 
    \begin{center}
        $\chi_{dom}(S_m)=m$
    \end{center}
\end{corollary}
\begin{proof}
The sun graph is a special case of the graph $S_m^n$ in which $n=1$.
\end{proof}
\begin{corollary}
    Let $G$ be a graph formed by attaching $k$ pendant vertices to each  vertex of a cycle graph $C_m$. Then the dominated chromatic number of $G$ is given by
    \begin{center}
        $\chi_{dom}(G)=m$
    \end{center}
\end{corollary}
\begin{proof}
    The graph $G$ is a special case of the graph considered in the previous theorem obtained by taking $n=1$.
\end{proof}
\begin{theorem}
    Let $G$ be a generalized sun graph formed by attaching exactly two leaves to a cycle graph $C_n$ of order $n$, and let $d$ be the distance between the vertices to which the leaves are attached. Then the dominated chromatic number of G is,
    \begin{center}
        $\chi_{dom}(G)=\begin{cases}
           \frac{n}{2}+1, & \text{if} \ n\equiv0\pmod{4} \ \text{and}\ d=2\\
           \frac{n}{2}, & \text{if}\ n\equiv0\pmod{4}\ \text{and}\ d\neq2\\
           \lfloor\frac{n}{2}\rfloor+1, & \text{if}\ n\not\equiv0\pmod{4}.
        \end{cases}$
    \end{center}
\end{theorem} 
\begin{proof}
Let
\begin{center}
    $\{x_1,x_2,.\ .\ .\ ,x_n\} \cup \{y_1,y_2\} $
\end{center}
be the vertices of $G$ where $x_i's$ spans the cycle $C_n$ and $y_1$ and $y_2$ are the pendant vertices attached to $x_i$ and $x_j$ for some $i$ and $j (1\leq i,j\leq n)$.

Since $G$ contains $C_n$ as a subgraph,
\begin{center}
    $\chi_{dom}(G)\geq\chi_{dom}(C_n)$.
\end{center}
We consider different possible cases for $n$ and $d$ to prove the theorem.
The vertices $y_1,x_{i-1}$ and $x_{i+1}$ forms a $K_{1,3}$ while the vertices $y_2,x_{j-1}$ and $x_{j+1}$ forms a $K_{1,3}$.\\

\noindent Case (1): $d\neq2$ \\
For any $n$, we can reuse the colors of $x_{i-1}$ and $x_{j+1}$ to the vertices $y_1$ and, $y_2$ respectively. Therefore, we can color $G$ using exactly $\chi_{dom}(C_n)$ colors. \\

\noindent Case (2): $d=2$ \\
Unless $n\equiv0\pmod{4}$, we can color $G$ using exactly $\chi_{dom}(C_n)$ colors by reusing colors as in the first case. When $n\equiv0\pmod{4}$, in an optimal dominated coloring of $C_n$, each color class contains exactly two vertices. Therefore, for any one of the vertices $y_1,y_2$ , we can reuse the color in $C_n$ and the other one should receive an additional color. In this case the dominated coloring of $G$ require $\chi_{dom}(G)+1$ colors.\\
Hence the dominated chromatic number of $G$ is given by,\\
\begin{center}
   $\chi_{dom}(G)=\begin{cases}
       \chi_{dom}(C_n)+1,& \text{if}\ d=2\ \text{and}\ n\equiv0\pmod{4}\\
       \chi_{dom}(C_n),& \text{otherwise}.
   \end{cases}$
\end{center}
Theorem\ref{uni2} gives the desired result.
\end{proof}

\bibliographystyle{plain}
\bibliography{references}
\end{document}